\documentclass[preprint,12pt]{elsarticle}

\usepackage{amssymb}
 \usepackage{amsthm,amsmath}
\usepackage{mathrsfs}

 \numberwithin{equation}{section}

\newtheorem{thm}{Theorem}[section]
\newtheorem{lem}[thm]{Lemma}

\theoremstyle{definition}
\newtheorem{rem}[thm]{Remark}

\allowdisplaybreaks

\begin{document}

\begin{frontmatter}



\title{The regularity criterion  for 3D Navier-Stokes
Equations involving  one   velocity  gradient  component}


\author{Daoyuan Fang}
\ead{dyf@zju.edu.cn}
\author{Chenyin Qian}
\ead{qcyjcsx@163.com}
\address{Department of Mathematics, Zhejiang University, Hangzhou 310027, China}
\begin{abstract}
In this article, we establish sufficient conditions for the
regularity of solutions of Navier-Stokes equations based on one of
the nine entries of the gradient tensor. We improve the recently
results of C.S. Cao, E.S. Titi (Arch. Rational Mech.Anal. 202 (2011)
919-932) and Y. Zhou, M. Pokorn$\acute{\mbox{y} }$ (Nonlinearity 23,
1097-1107 (2010)).
\end{abstract}


\end{frontmatter}

\section{Introduction}
We consider sufficient conditions for the regularity of  weak
solutions of the Cauchy problem for the Navier-Stokes equations
\begin{equation} \label{a}
 \left\{\begin{array}{l}
\displaystyle \frac{\partial u}{\partial t}-\nu \Delta
u+(u\cdot\nabla)u+\nabla p=0,
\ \mbox{\ in}\ \mathbb{R}^{3}\times (0,T),\\
\displaystyle\nabla\cdot u=0,\hspace{3.52cm} \mbox{\ in}\ \mathbb{R}^{3}\times (0,T),\\
\displaystyle u(x, 0)=u_{0},\hspace{0.2cm}\mbox{\ in}\ \mathbb{R}^{3},\\
\end{array}
\right. \end{equation} where $u=(u_{1},u_{2},u_{3}):
\mathbb{R}^{3}\times (0,T)\rightarrow \mathbb{R}^{3}$ is the
velocity field, $ p: \mathbb{R}^{3}\times (0,T)\rightarrow
\mathbb{R}^{3}$ is a scalar pressure, and $u_{0}$ is the initial
velocity field, $\nu>0$ is the viscosity.  We set
$\nabla_{h}=(\partial_{x_{1}},\partial_{x_{2}})$ as the horizontal
gradient operator and
$\Delta_{h}=\partial_{x_{1}}^{2}+\partial_{x_{2}}^{2}$ as the
horizontal Laplacian, and $\Delta$ and $\nabla$ are the usual
Laplacian and the gradient operators, respectively.  Here we use the
classical notations
$$
(u\cdot\nabla)v=\sum_{i=1}^{3}u_{i}\partial_{x_{i}}v_{k}, \ (
k=1,2,3),\ \ \ \nabla\cdot u=\sum_{i=1}^{3}\partial_{x_{i}}u_{i},
$$
 and for sake of simplicity,  we denote $\partial_{x_{i}}$
by $\partial_{{i}}$.
\par We set
$$
\mathcal {V}=\{\phi: \ \mbox{ the 3D vector valued}\ C_{0}^{\infty}
\ \mbox{functions and}\ \nabla\cdot\phi=0\},
$$
which will form the space of test functions. Let $H$ and $V$ be the
closure spaces of $\mathcal {V}$ in $L^2$ under $L^2$-topology, and
in $H^1$ under $H^1$-topology, respectively.
\par
For $u_0\in H$,  the existence of weak solutions of (1.1) was
established by Leray  \cite{[14]} and Hopf in  \cite{[9]}, that is,
$u$
satisfies the following properties:\\
(i) $u\in C_{w}([0,T); H)\cap L^{2}(0,T; V)$, and $\partial_{t}u\in
L^{1}(0,T; V^{\prime})$, where $V^{\prime}$ is the dual space of
$V$;\\
(ii) $u$ verifies (1.1) in the sense of distribution, i.e., for
every test function $\phi\in C^{\infty}([0,T);\mathcal {V})$, and
for almost every $t, t_{0}\in (0,T)$, we have
$$
\begin{array}{ll}
 \displaystyle&\displaystyle\int_{\mathbb{R}^{3}} u(x,t)\cdot\phi(x,t)dx-\int_{\mathbb{R}^{3}}
u(x,t_{0})\cdot\phi(x,t_{0})dx\vspace{2mm}\\
\displaystyle &\ \ \ \
=\displaystyle\int_{t_{0}}^{t}\int_{\mathbb{R}_{3}}[u(x,t
)\cdot(\phi_{t}(x,t)+\nu\Delta\phi(x,t))]dxds\vspace{2mm}\\
&\ \ \ \ \ \ \
+\displaystyle\int_{t_{0}}^{t}\int_{\mathbb{R}_{3}}[(u(x,t)\cdot\nabla)\phi(x,t)]\cdot
u(x,t))]dxds;
\end{array}
$$
 (iii) The energy inequality,
i.e.,
$$
\|u(\cdot,t)\|_{L^{2}}^{2}+2\nu\int_{t_0}^{t}\|\nabla
u(\cdot,s)\|_{L^{2}}^{2}ds\leq\|u_{0}\|_{L^{2}}^{2},
$$
for every $t$ and almost every $t_{0}$.

It is well known, if $u_{0}\in V$, a weak solution
of (1.1) on $(0, T)$ becomes  strong if it satisfies
$$
u\in C([0,T); V)\cap L^{2}(0,T; H^{2}) \ \mbox{and}\
\partial_{t}u\in L^{2}(0,T; H).
$$
We know the strong solution is regular(say, classical) and unique
(see, for example, \cite{[21]}, \cite{[22]}).
\par    For the 2D case, just as the authors  said in \cite{[3]},
 the Navier-Stokes equations \eqref{a} have unique weak and
strong solutions which exist globally in time. However, the global
regularity of solutions  for the 3D Navier-Stokes equations is a
major and challenging problem, the weak solutions are known to exist
globally in time, but the uniqueness, regularity, and continuous
dependence on initial data for weak solutions are still open
problems. Furthermore, strong solutions in the 3D case are known to
exist for a short interval of time whose length depends on the
initial data. Moreover, this strong solution is known to be unique
and to depend continuously on the initial data.
\par There are many  interesting sufficient conditions which guarantee
that a given weak solution is smooth (see, for example,
\cite{[4]}-\cite{[19]}), and the first result is usually referred as
Prodi-Serrin conditions (see \cite{[18]} and \cite{[20]}), which
states that if a weak solution $u$ is in the class of
$$ u\in L^{t}(0,T; L^{s}(\mathbb{R}^{3})),\ \
\frac{2}{t}+\frac{3}{s}=1,\ s\in[3,\infty],
$$ then the weak solution becomes regular.
\par A better result was showed by  Neustupa, Novotny, and Penel
(see  \cite{[16]}). More precisely, the solution is regular if
$$ u_{3}\in
L^{t}(0,T; L^{s}(\mathbb{R}^{3})),\ \
\frac{2}{t}+\frac{3}{s}\leq\frac{1}{2},\ s\in(6,\infty].
$$
This result was improved in \cite{[10]} to
$$ u_{3}\in
L^{t}(0,T; L^{s}(\mathbb{R}^{3})),\ \
\frac{2}{t}+\frac{3}{s}\leq\frac{5}{8},\ s\in(\frac{24}{5},\infty].
$$
C.S. Cao, E.S. Titi in \cite{[3]} considered the regularity of
solutions to the  3D Navier-Stokes equations subject to periodic
boundary conditions or in the whole space and obtained  better
results
$$
 u_{3}\in
L^{t}(0,T; L^{s}(\mathbb{R}^{3})),\ \
\frac{2}{t}+\frac{3}{s}\leq\frac{2(s+1)}{3s},\ s>\frac{7}{2},
$$
or
$$
 u_{3}\in
L^{\infty}(0,T; L^{s}(\mathbb{R}^{3})),\ \mbox{with} \
s>\frac{7}{2}.
$$
Furthermore, this work was improved by Y. Zhou, M.
Pokorn$\acute{\mbox{y} }$ in \cite{[26]}, the authors considered the
following additional assumptions to get the regularity of solution
of 3D Navier-Stokes equations
$$ u_{3}\in
L^{t}(0,T; L^{s}(\mathbb{R}^{3})),\ \
\frac{2}{t}+\frac{3}{s}\leq\frac{3}{4}+\frac{1}{2s},\
s>\frac{10}{3}.
$$ The
full regularity of weak solutions can also be proved under
alternative assumptions on the gradient of the velocity $\nabla u$.
Specifically (see \cite{[1]}), if
$$
\nabla u\in L^{t}(0,T; L^{s}(\mathbb{R}^{3})),\ \
\frac{2}{t}+\frac{3}{s}=2,\ s\in[\frac{3}{2},\infty].
$$
A comparable result for the gradient of one velocity component was
improved in \cite{[17]} to
$$
\nabla u_{3}\in L^{t}(0,T; L^{s}(\mathbb{R}^{3})),\ \
\frac{2}{t}+\frac{3}{s}\leq\frac{3}{2},\ s\in[{2},\infty].
$$
There are many  similar results, we refer to the references
\cite{[27]}. In  \cite{[26]}, the authors also studied the
regularity of the solutions of the Navier-Stokes equations under the
assumption on $\partial_{3}u_{3}$, namely,
\begin{equation}\label{bb}
\partial_{3}u_{3}\in L^{\beta}(0,T; L^{\alpha}(\mathbb{R}^{3})),\
\frac{3}{\alpha}+\frac{2}{\beta}< \frac{4}{5},\
\alpha\in(\frac{15}{4}, \infty].
\end{equation}
\par
Very recently, C.S. Cao and E.S. Titi considered the more general case
in \cite{[2]}, in which authors  provided sufficient conditions, in
terms of only one of the nine components of the gradient of velocity
field (i.e., the velocity Jacobian matrix) that guarantee the global
regularity of the 3D Navier-Stokes equations. The authors divided
into cases to discuss the regularity of the weak solution, namely,
given the condition
$$
\frac{\partial u_{j}}{\partial x_{k}}\in L^{\beta}(0,T;
L^{\alpha}(\mathbb{R}^{3})), \ \mbox{when} \ j\neq k
$$
\begin{equation}\label{c}
\mbox{and where}  \ \alpha>3, 1\leq\beta<\infty, \mbox{and}\
\frac{3}{\alpha}+\frac{2}{\beta}\leq\frac{\alpha+3}{2\alpha},\end{equation}
or
$$
\frac{\partial u_{j}}{\partial x_{j}}\in L^{\beta}(0,T;
L^{\alpha}(\mathbb{R}^{3})),
$$
\begin{equation}\label{d}
\mbox{and where}  \ \alpha>2, 1\leq\beta<\infty, \mbox{and} \
\frac{3}{\alpha}+\frac{2}{\beta}\leq\frac{3(\alpha+2)}{4\alpha}.\end{equation}
Moreover, Z.J. Zhang studied the Cauchy problem for the 3D
Navier-Stokes equations, and proved some scalaring-invariant
regularity criteria involving only one velocity component in
\cite{[31]}. The author proved that the weak solution $u$
 to \eqref{a} with  datum  $u_0\in V$ is regular, if
\begin{equation}\label{ff1} u_3\in L^{p}(0,T; L^{q}(\mathbb{R}^{3})),
\partial_3u_3\in L^{\beta}(0,T; L^{\alpha}(\mathbb{R}^{3})),
\end{equation}
with $1\leq p, q, \beta, \alpha\leq\infty$, $0\leq\lambda,
\gamma<\infty$ satisfying
\begin{equation} \label{ff3}
 \left\{\begin{array}{l}
\displaystyle \frac{2}{p}+\frac{3}{q}=\lambda,
\frac{2}{\beta}+\frac{3}{\alpha}=\gamma,\vspace{2mm}\\
\displaystyle
(1-\frac{1}{\alpha})q=\frac{1/\beta+3/8}{3/8-1/p}=\frac{9/4-\gamma}{\lambda-3/4
}>1,\vspace{2mm}\\
\displaystyle \ \beta<\infty \ \mbox{or}\  p<\infty.\\
\end{array}
\right. \end{equation}

Motivated by \cite{[2]} and \cite{[31]}, in this article, we
consider the alternative assumptions on one velocity gradient
component, and we improve the results of  \cite{[2]}. The key point
of our approach is that we start with the estimate of the norm
$\|u_{3}\|_{q},$ where $q$ satisfies $q\geq2$, and then construct
some new estimates. We also improve the result of \cite{[26]}. From
our argument, one can know which cause the difference of the both
results in \cite{[26]} and \cite{[2]}.
 \par Our main results  can be stated in the
following:
\begin{thm}\label{t1.1}
Let  $u_{0}\in V$, and  assume $u$ is a Leray-Hopf weak solution to
the 3D Navier-Stokes equations \eqref{a}. Suppose for any $j, k$ with $1\leq j, k\leq 3$, we have
\begin{eqnarray}\label{e} \partial_{j}u_{k}\in L^{\infty}(0,T;
L^3(\mathbb{R}^3)).
\end{eqnarray}
 Then $u$ is
regular.
\end{thm}

\begin{thm}\label{t1.3}
Let  $u_{0}$ and  $u$ be as   in  Theorem \ref{t1.1}. For any $j, k$
with
$1\leq j, k\leq 3$.\\
 (i) For $ j\neq k$,  suppose that $u$ satisfies
\begin{eqnarray}\label{z} \int_{0}^{T}\|
\partial_{j}u_{k}\|_{\alpha}^{\beta}d\tau\leq M, \ \mbox{for some} \ M>0, \end{eqnarray}
with
\begin{equation} \label{f}
 \frac{3}{2\alpha}+\frac{2}{\beta}\leq
\displaystyle f(\alpha), \  \alpha\in (3,\infty)\ \mbox {and}\
1\leq\beta<\infty,  \end{equation} where
\begin{equation}\label{m}
f(\alpha)=\frac{\sqrt{103\alpha^2-12\alpha+9}-9\alpha}{2\alpha};
\end{equation}
(ii)For $j=k$, suppose
\begin{eqnarray}\label{zz1} \int_{0}^{T}\|
\partial_{k}u_{k}\|_{\alpha}^{\beta}d\tau\leq M, \ \mbox{for some} \ M>0,\end{eqnarray}
with
\begin{equation} \label{zz2}
 \frac{3}{2\alpha}+\frac{2}{\beta}\leq
\displaystyle g(\alpha), \  \frac{9}{5}<\alpha<\infty\ \mbox {and}\
1\leq\beta<\infty,\end{equation} where
$$
g(\alpha)=\frac{\sqrt{289\alpha^2-264\alpha+144}-7\alpha}{8\alpha}.
$$
\par Then $u$ is  regular.
\end{thm}

\begin{rem}\label{r1.5} Theorem \ref{t1.1} give us an
endpoint version of regularity criterion, which is a complement of
\cite{[2]}. From the proof of it, one can know this result in fact
should have been included in \cite{[2]}.   Compared
with the results of \cite{[2]}, it is easy to
check that Theorem \ref{t1.3} (i) is an improvement of \eqref{c}
(see {{Figure 1}}). For Theorem \ref{t1.3} (ii), the allowed region of $(\alpha, \beta)$ in our rsult is much
large than those of \cite{[26]} and \cite{[2]} (see {Figure 2}).
\end{rem}
\par For the convenience, we recall the following version of
the three-dimensional  Sobolev and Ladyzhenskaya inequalities in the
whole space $\mathbb{R}^{3}$ (see, for example, \cite{[5]}
-\cite{[11]}). There exists a positive constant $C$ such that
\begin{equation}\label{i}
\begin{array}{ll}
 \displaystyle
\|u\|_{r}&\displaystyle
\leq C \|u\|_{2}^{\frac{6-r}{2r}}\|\partial_{1}u\|_{2}^{\frac{r-2}{2r}}\|\partial_{2}u\|_{2}^{\frac{r-2}{2r}}\|\partial_{3}u\|_{2}^{\frac{r-2}{2r}}\\
&\leq C \|u\|_{2}^{\frac{6-r}{2r}}\|\nabla
u\|_{2}^{\frac{3(r-2)}{2r}},
\end{array}
\end{equation}
for every $u\in H^{1}(\mathbb{R}^{3})$ and every $r\in[2,6]$, where
$C$ is a constant depending only on $r$. Taking $\nabla$div on both
sides of \eqref{a} for smooth $(u, p)$, one can obtain
$$
-\Delta(\nabla
p)=\sum_{i,j}^{3}\partial_{i}\partial_{j}(\nabla(u_{i}u_{j})),
$$
therefore, the Calderon-Zygmund inequality in $\mathbb{R}^{3}$ (see
\cite{[30]})
\begin{equation}\label{j}
\|\nabla p\|_{q}\leq C\||\nabla u||u|\|_{q},
 1<q<\infty,
\end{equation}
holds, where $C$ is a positive constant depending only on $q$. And
there  is another  estimate for  pressure
\begin{equation}\label{p} \| p\|_{q}\leq C\|u\|_{2q}^2, \
 1<q<\infty.
\end{equation}
 Let
$$ 1\leq q_{i}<\infty,\  i=1,...,n.
$$
Then,  for all $u\in C^{\infty}_{0}(\mathbb{R}^{n})$ the following
Troisi inequality holds (see \cite{[8]}):
\begin{equation}\label{k}
\|u\|_{s}\leq
C\prod_{i=1}^{n}\|D_{i}u\|_{q_{i}}^{\frac{1}{n}},\end{equation}
where $q_{i}, n$ and $s$ satisfy
\begin{equation*}\label{b}
\sum_{i=1}^{n}q_{i}^{-1}>1, \ \mbox{and}\
s=\frac{n}{\sum_{i=1}^{n}q_{i}^{-1}-1}.\end{equation*}

\section{ \textit{A Priori} Estimates}
In this section, under the assumptions of Theorem \ref{t1.1},
\ref{t1.3}, we will prove some \textit{a priori} estimates, which
are needed  in the proof of our results. First of all, we
note that the Leray-Hopf weak solutions have the energy inequality
(see, for example, \cite{[21]}, \cite{[22]},
 \cite{[5]} for detail)
\begin{equation}\label{1}
\|u(\cdot,t)\|_{L^{2}}^{2}+2\nu\int_{0}^{t}\|\nabla
u(\cdot,s)\|_{L^{2}}^{2}ds\leq K_{1},\end{equation} for all $0<
t<T,$ where $K_{1}=\|u_{0}\|_{L^{2}}^{2}.$
\par Then,  an estimate of $ \|u_{3}\|_{q}$  can be read in
the following lemma:
\begin{lem}\label{l2.1}
Assume that
\begin{equation}\label{2}
3\leq{{\alpha}}<\infty,\  \ 1<\sigma\leq \frac{9}{8}, \ \mbox{and}\
q\geq2,\end{equation} where $q, {\alpha},$ and $ \sigma$ satisfy
\begin{equation}\label{3}
\frac{1}{\sigma}+\frac{q-2}{q}+\frac{1}{3{\alpha}}=1,\end{equation}
then we have the following estimates
\begin{equation}\label{4}
\frac{1}{2}\frac{d}{dt}\|u_{3}\|_{q}^{2}\leq C \|\nabla
u\|_{2}^{\frac{8}{3}-s}\|\partial_{1}u_{3}\|_{{\alpha}}^{{1}/{3}}, \
\mbox{with}\ \ s=\frac{3-2\sigma}{\sigma}.
\end{equation}
\end{lem}
\begin{proof}
 We use $|u_{3}|^{q-2}u_{3},$ $ q\geq2,$ as test function  in
the equation \eqref{a} for $u_{3}.$ By using of Gagliardo-Nirenberg,
H$\ddot{\mbox{o}}$lder's inequalities, \eqref{j} and \eqref{k}, we
have
\begin{equation}\label{5}
\begin{array}{ll}
 \displaystyle \frac{1}{q}\frac{d}{dt}\|u_{3}\|_{q}^{q}&+C(q)\nu\|\nabla|u_{3}|^{\frac{q}{2}}\|_{2}^{2}\\
 &=\displaystyle
 -\int_{\mathbb{R}^{3}}\partial_{3}p|u_{3}|^{q-2}u_{3}dx\displaystyle  \vspace{2mm}\\
 \displaystyle &\leq\displaystyle
  \|\partial_{3}p\|_{\sigma}\|u_{3}\|_{q}^{q-2}\|u_{3}\|_{3{\alpha}}\displaystyle \vspace{2mm}\\
 \displaystyle &\leq\displaystyle   C \||\nabla u| |u|
 \|_{\sigma}\|u_{3}\|_{q}^{q-2}\|u_{3}\|_{3\alpha}\displaystyle
 \ (\mbox{by }\eqref{j})\vspace{2mm}\\
 \displaystyle &\leq\displaystyle C\|\nabla u\|_{2}\| u\|_{\frac{2\sigma}{2-\sigma}}
 \|u_{3}\|_{q}^{q-2}\|\partial_{1}u_{3}\|_{\alpha}^{{1}/{3}}
 \|\partial_{2}u_{3}\|_{2}^{{1}/{3}}\|\partial_{3}u_{3}\|_{2}^{{1}/{3}}
 \vspace{2mm}\\
\displaystyle &\leq\displaystyle C\|\nabla u\|_{2}\|
u\|_{2}^{s}\|\nabla u\|_{2}^{1-s}
 \|u_{3}\|_{q}^{q-2}\|\partial_{1}u_{3}\|_{\alpha}^{{1}/{3}}
 \|\nabla u\|_{2}^{{2}/{3}}\vspace{2mm}\\
\displaystyle &=\displaystyle C\| u\|_{2}^{s}\|\nabla
u\|_{2}^{\frac{8}{3}-s}\|u_{3}\|_{q}^{q-2}\|\partial_{1}u_{3}\|_{\alpha}^{{1}/{3}}.
\end{array}
\end{equation}
For the index in above inequality, we know
$1<\sigma\leq\frac{9}{8}<\frac{3}{2}$,
$$
2<\frac{2\sigma}{2-\sigma}< 6,
$$
 and $s$ satisfies
$$
\frac{2-\sigma}{2\sigma}=\frac{s}{2}+ \frac{1-s}{6}, \
\mbox{namely,}\ \ s=\frac{3-2\sigma}{\sigma},\ \  by \ \eqref{2}.
$$
In view of \eqref{1}, \eqref{5} implies that
$$
\frac{1}{2}\frac{d}{dt}\|u_{3}\|_{q}^{2}\leq C \|\nabla
u\|_{2}^{\frac{8}{3}-s}\|\partial_{1}u_{3}\|_{\alpha}^{{1}/{3}}.
$$The proof is thus completed.
\end{proof}
\par Next, we estimate \ $ \|\nabla_{h}u\|_{2}$:
\begin{lem}\label{l2.2}
Assume that ${\alpha}$ and $q$  satisfy the conditions in Lemma
\ref{l2.1}. Set
\begin{equation}\label{6}
r=\frac{(q+1)\alpha-q}{\alpha},
\end{equation}
then we have the following estimates
\begin{equation}\label{7}
\begin{array}{ll}
 \displaystyle &\|\nabla_{h}u\|_{2}^{2}+\displaystyle\nu\int_{0}^{t}\|\nabla_{h}\nabla
u\|_{2}^{2}d\tau
 \displaystyle\leq
\displaystyle\|\nabla_{h} u(0)\|_{2}^{2}\\& \ \ \ \ \ \ \ \ \ \
+C\displaystyle[\int_{0}^{t}\|u_{3}\|_{q}^{\frac{2(r-1)}{r-2}}\|
\partial_{1}u_{3}\|_{\alpha}^{\frac{2}{r-2}}\|\nabla
u\|_{2}^{2}d\tau]^{\frac{r-2}{r-1}}\times
 [\int_{0}^{t}\|\Delta u\|_{2}^{2}d\tau]^{\frac{1}{r-1}},\vspace{2mm}\\
 \end{array}\end{equation}
and
\begin{equation}\label{100}
\begin{array}{ll}
&\|\nabla_{h}u\|_{2}^{2}+\displaystyle\nu\int_{0}^{t}\|\nabla_{h}\nabla
u\|_{2}^{2}d\tau\\
&\ \ \ \ \ \ \ \ \leq \displaystyle\|\nabla_{h}
u(0)\|_{2}^{2}+C\displaystyle\int_{0}^{t}\|u_{3}\|_{q}^{\frac{2(r-1)}{r-2}}\|
\partial_{3}u_{3}\|_{\alpha}^{\frac{2}{r-2}}\|\nabla
u\|_{2}^{2}d\tau.
\end{array}\end{equation}
\end{lem}
\begin{proof}
Taking the inner product of the equation \eqref{a} with
$-\Delta_{h}u$ in $L^{2}$,  applying H$\ddot{\mbox{o}}$lder's
inequality several times, we obtain
\begin{equation}\label{8}
\begin{array}{ll}
 &\displaystyle \frac{1}{2}\frac{d}{dt}\|\nabla_{h}u\|_{2}^{2}+\nu\|\nabla_{h}\nabla u\|_{2}^{2}\\
 &\ \ \ \ =\displaystyle
 \int_{\mathbb{R}^{3}}[(u\cdot \nabla)u]\Delta_{h}u dx\displaystyle \\
 \displaystyle &\ \ \ \ \leq\displaystyle
 C\int_{\mathbb{R}^{3}}|u_{3}||\nabla u||\nabla_{h}\nabla u|dx \ \ (\mbox{see}\  \mbox{\cite{[2]}})\displaystyle \\
 \displaystyle &\ \ \ \ \leq\displaystyle
 C\int_{\mathbb{R}^{2}}\max_{x_{1}}|u_{3}|(\int_{\mathbb{R}}|\nabla u|^{2}dx_{1})^{\frac{1}{2}}(\int_{\mathbb{R}}|\nabla_{h}\nabla u|^{2}dx_{1})^{\frac{1}{2}} dx_{2}dx_{3}\
 \displaystyle \vspace{2mm}\\
 \displaystyle &\ \ \ \ \leq\displaystyle
 C[\int_{\mathbb{R}^{2}}(\max_{x_{1}}|u_{3}|)^{r}dx_{2}dx_{3}]^{\frac{1}{r}}[\int_{\mathbb{R}^{2}}(\int_{\mathbb{R}}|\nabla
 u|^{2}dx_{1})^{\frac{r}{r-2}}dx_{2}dx_{3}]^{\frac{r-2}{2r}}\\
 \displaystyle& \ \ \ \ \ \ \times \displaystyle[\int_{\mathbb{R}^{3}}|\nabla_{h}\nabla
 u|^{2}dx_{1}dx_{2}dx_{3}]^{\frac{1}{2}}\\
 \displaystyle \vspace{2mm}
 \displaystyle &\ \ \ \ \leq\displaystyle
 C[\int_{\mathbb{R}^{3}}|u_{3}|^{r-1}\partial_{1}u_{3}dx_{1}dx_{2}dx_{3}]^{\frac{1}{r}}\|\nabla_{h}\nabla u\|_{2}
  \displaystyle \\
 \displaystyle& \ \ \ \ \ \ \times \displaystyle[\int_{\mathbb{R}}(\int_{\mathbb{R}^{2}}|\nabla
 u|^{\frac{2r}{r-2}}dx_{2}dx_{3})^{\frac{r-2}{r}}dx_{1}]^{\frac{1}{2}}
 \displaystyle \vspace{2mm}\\
\displaystyle &\ \ \ \ \ \leq \displaystyle C\displaystyle\|
u_{3}\|_{q}^{\frac{r-1}{r}}\|
\partial_{1}u_{3}\|_{\frac{q}{q-r+1}}^{\frac{1}{r}}\|\nabla u\|_{2}^{\frac{r-2}{r}}
 \|\partial_{2}\nabla u\|_{2}^{\frac{1}{r}} \|\partial_{3}\nabla u\|_{2}^{\frac{1}{r}}
\|\nabla_{h}\nabla u\|_{2}.\vspace{2mm}\\
\end{array}\end{equation}
 By \eqref{6}, and the fact $\alpha\geq 3, $ $q\geq2$, we have
\begin{equation}\label{9}
  \frac{q}{q-r+1}=\alpha,\ \mbox{and}\ \frac{7}{3}\leq r<q+1.
\end{equation}
 To prove \eqref{7}, applying  Young's inequality to \eqref{8}, we obtain
$$
\begin{array}{ll}
 &\displaystyle \frac{1}{2}\frac{d}{dt}\|\nabla_{h}u\|_{2}^{2}+\nu\|\nabla_{h}\nabla u\|_{2}^{2}\vspace{2mm}\\
\displaystyle &\hspace{1cm}\leq C\displaystyle\| u_{3}\|_{q}^{2}\|
\partial_{1}u_{3}\|_{\alpha}^{\frac{2}{r-1}}\|\nabla u\|_{2}^{\frac{2(r-2)}{r-1}}
 \|\Delta u\|_{2}^{\frac{2}{r-1}} +\frac{\nu}{2}
\|\nabla_{h}\nabla u\|_{2}^{2}.\vspace{2mm}\\
\end{array}
$$
Absorbing the last term in right hand side and integrating the above
inequality, using H$\ddot{\mbox{o}}$lder's inequality, we have
$$
\begin{array}{ll}
 \displaystyle &\|\nabla_{h}u\|_{2}^{2}+\displaystyle\nu\int_{0}^{t}\|\nabla_{h}\nabla
u\|_{2}^{2}d\tau\vspace{2mm}\\
 \displaystyle &\ \ \ \ \leq
\displaystyle\|\nabla_{h}
u(0)\|_{2}^{2}+C\displaystyle\int_{0}^{t}\|u_{3}\|_{q}^{2}\|
\partial_{1}u_{3}\|_{\alpha}^{\frac{2}{r-1}}\|\nabla u\|_{2}^{\frac{2(r-2)}{r-1}}
 \|\Delta u\|_{2}^{\frac{2}{r-1}}d\tau\vspace{2mm}\\
 \displaystyle &\ \ \ \ \leq
\displaystyle\|\nabla_{h}
u(0)\|_{2}^{2}+C\displaystyle[\int_{0}^{t}\|u_{3}\|_{q}^{\frac{2(r-1)}{r-2}}\|
\partial_{1}u_{3}\|_{\alpha}^{\frac{2}{r-2}}\|\nabla
u\|_{2}^{2}d\tau]^{\frac{r-2}{r-1}}\times
 [\int_{0}^{t}\|\Delta u\|_{2}^{2}d\tau]^{\frac{1}{r-1}}.\vspace{2mm}\\
 \end{array}
$$
To prove \eqref{100}, firstly,  we note that we can get a similar
inequality to \eqref{8} as follows
\begin{equation}\label{99}
\begin{array}{ll}
 &\displaystyle \frac{1}{2}\frac{d}{dt}\|\nabla_{h}u\|_{2}^{2}+\nu\|\nabla_{h}\nabla u\|_{2}^{2}\\
\displaystyle &\ \ \ \ \ \leq \displaystyle C\displaystyle\|
u_{3}\|_{q}^{\frac{r-1}{r}}\|
\partial_{3}u_{3}\|_{\alpha}^{\frac{1}{r}}\|\nabla u\|_{2}^{\frac{r-2}{r}}
 \|\partial_{2}\nabla u\|_{2}^{\frac{1}{r}} \|\partial_{2}\nabla u\|_{2}^{\frac{1}{r}}
\|\nabla_{h}\nabla u\|_{2}.\vspace{2mm}\\
\end{array}\end{equation}
Applying Young's inequality to \eqref{99}, we have
\begin{equation}\label{98}
\begin{array}{ll}
 &\displaystyle \frac{1}{2}\frac{d}{dt}\|\nabla_{h}u\|_{2}^{2}+\nu\|\nabla_{h}\nabla u\|_{2}^{2}\\
\displaystyle &\ \ \ \ \ \leq \displaystyle C\displaystyle\|
u_{3}\|_{q}^{\frac{2(r-1)}{r-2}}\|
\partial_{3}u_{3}\|_{\alpha}^{\frac{2}{r-2}}\|\nabla u\|_{2}^{2}+\frac{\nu}{2}\|\nabla_{h}\nabla u\|_{2}^{2}.\vspace{2mm}\\
\end{array}\end{equation}
As above, absorbing the last term in right hand side of \eqref{98} and
integrating the above inequality, using H$\ddot{\mbox{o}}$lder's
inequality, we obtain
\begin{equation*}
\begin{array}{ll}
\|\nabla_{h}u\|_{2}^{2}+\displaystyle\nu\int_{0}^{t}\|\nabla_{h}\nabla
u\|_{2}^{2}d\tau\leq \displaystyle\|\nabla_{h}
u(0)\|_{2}^{2}+C\displaystyle\int_{0}^{t}\|u_{3}\|_{q}^{\frac{2(r-1)}{r-2}}\|
\partial_{3}u_{3}\|_{\alpha}^{\frac{2}{r-2}}\|\nabla
u\|_{2}^{2}d\tau
\end{array}\end{equation*}
The proof of Lemma \ref{l2.2} is completed.
\end{proof}
\par At last, we estimate \ $ \|\nabla u\|_{2} :$
\begin{lem}\label{l2.3} Let
${\alpha}$, $q$ and $r$ satisfy Lemma \ref{l2.1} and Lemma
\ref{l2.2}, then we have the
following estimates:\\
(i),
\begin{equation}\label{10} \begin{array}{ll}
&\|\nabla
u\|_{2}^{2}+\displaystyle{\nu}\displaystyle\int_{0}^{t}\|\Delta
u\|_{2}^{2}d\tau\\&\hspace{1cm}\leq\displaystyle\displaystyle
C\int_{0}^{t}\|u_{3}\|_{q}^{\frac{8(r-1)}{3r-7}}\|
\partial_{1}u_{3}\|_{\alpha}^{\frac{8}{3r-7}}\|\nabla
u\|_{2}^{2}d\tau+\|\nabla u(0)\|_{2}^{2}\\
&\hspace{1cm}\hspace{0.5cm}+\displaystyle
C\int_{0}^{t}\|u_{3}\|_{q}^{\frac{2(r-1)}{r-2}}\|
\partial_{1}u_{3}\|_{\alpha}^{\frac{2}{r-2}}\|\nabla
u\|_{2}^{2}d\tau+C
\end{array}\end{equation}if $3r-7>0$.\\
(ii),
\begin{equation}\label{11} \begin{array}{ll}
&\|\nabla
u\|_{2}^{2}+\displaystyle\frac{5\nu}{4}\displaystyle\int_{0}^{t}\|\Delta
u\|_{2}^{2}d\tau\\
\displaystyle &\hspace{1cm} \leq\|\nabla
u(0)\|_{2}^{2}+\displaystyle C\int_{0}^{t}\|u_{3}\|_{q}^{8}\|
\partial_{1}u_{3}\|_{\alpha}^{6}\|\nabla
u\|_{2}^{2}d\tau+ C\\ &\hspace{1cm}\hspace{0.5cm}
+C\displaystyle[\int_{0}^{t}\|u_{3}\|_{q}^{8}\|
\partial_{1}u_{3}\|_{\alpha}^{6}\|\nabla
u\|_{2}^{2}d\tau]^{\frac{1}{4}}\times
 \int_{0}^{t}\|\Delta
 u\|_{2}^{2}d\tau
\end{array}\end{equation}if $3r-7=0$.
Moreover, we have
\begin{equation} \label{c95}\begin{array}{ll}
&\|\nabla
u\|_{2}^{2}+\displaystyle\frac{\nu}{2}\displaystyle\int_{0}^{t}\|\Delta
u\|_{2}^{2}d\tau\\
&\ \ \ \ \ \  \leq\|\nabla u(0)\|_{2}^{2}+\displaystyle
C\int_{0}^{t}\| u_{3}\|_{q}^{\frac{2(r-1)}{r-2}}\|
\partial_{3}u_{3}\|_{\alpha}^{\frac{2}{r-2}}\|\nabla
u\|_{2}^{2}d\tau\\
&\hspace{1.3cm} +C\displaystyle\int_{0}^{t}\|
u_{3}\|_{q}^{\frac{8(r-1)}{3(r-2)}}\|
\partial_{3}u_{3}\|_{\alpha}^{\frac{8}{3(r-2)}}\|\nabla
u\|_{2}^{2}d\tau+C
\end{array}\end{equation}
\end{lem}
\begin{proof} Taking the inner product of the equation \eqref{a} with
$-\Delta u$ in $L^{2}$ , we obtain
$$
\begin{array}{ll}
 \displaystyle \frac{1}{2}\frac{d}{dt}\|\nabla u\|_{2}^{2}&+\nu\|\Delta u\|_{2}^{2}\\
 &=\displaystyle
 \sum_{i,j,k=1}^{3}\int_{\mathbb{R}^{3}}u_{i}\partial_{i}u_{j}\partial_{kk}u_{j} dx\displaystyle \\
 \displaystyle &=\displaystyle
 \sum_{i,j,k=1}^{2}\int_{\mathbb{R}^{3}}u_{i}\partial_{i}u_{j}\partial_{kk}u_{j}dx + \sum_{i=1}^{3}
  \sum_{j=1}^{2}\int_{\mathbb{R}^{3}}u_{i}\partial_{i}u_{j}\partial_{33}u_{j}dx\displaystyle \\
 \displaystyle &\ \ \ \ +\displaystyle
  \sum_{j=1}^{2}
  \sum_{k=1}^{2}\int_{\mathbb{R}^{3}}u_{3}\partial_{3}u_{j}\partial_{kk}u_{j}dx+\sum_{i,k=1}^{3}\int_{\mathbb{R}^{3}}u_{i}\partial_{i}u_{3}\partial_{kk}u_{3}dx\
 \displaystyle \\
 \displaystyle &=\displaystyle
 I_{1}(t)+I_{2}(t)+I_{3}(t)+I_{4}(t).
 \end{array}$$
 The calculation is  similar to  Lemma 2.2 in \cite{[7]}, for the convenience of  readers, we show it
 below. By integrating by parts several times and using the incompressibility
condition, we get
$$
I_{1}(t)=\frac{1}{2}\sum_{i,j,k=1}^{2}\int_{\mathbb{R}^{3}}\partial_{i}u_{i}\partial_{k}u_{j}\partial_{k}u_{j}dx
-\sum_{i,j,k=1}^{2}\int_{\mathbb{R}^{3}}\partial_{k}u_{i}\partial_{i}u_{j}\partial_{k}u_{j}dx=I_{1}^{1}(t)+I_{1}^{2}(t).
$$
The terms $I_{1}^{1}(t),$ $ I_{1}^{1}(t), $ $I_{3}(t)$ and $I_{4}(t)
$ read as
$$
I_{1}^{1}(t)=-\frac{1}{2}\sum_{j,k=1}^{2}\int_{\mathbb{R}^{3}}\partial_{3}u_{3}\partial_{k}u_{j}\partial_{k}u_{j}dx,
$$
$$\begin{array}{ll}\displaystyle
I_{1}^{2}(t)&=\displaystyle-\sum_{i,j,k=1}^{2}\int_{\mathbb{R}^{3}}\partial_{k}u_{i}\partial_{i}u_{j}\partial_{k}u_{j}dx\\
&
=\displaystyle\displaystyle\int_{\mathbb{R}^{3}}\partial_{2}u_{1}\partial_{1}u_{2}\partial_{2}u_{2}dx
+\displaystyle\int_{\mathbb{R}^{3}}\partial_{1}u_{2}\partial_{2}u_{1}\partial_{1}u_{1}dx
+\displaystyle\int_{\mathbb{R}^{3}}\partial_{1}u_{1}\partial_{1}u_{2}\partial_{1}u_{2}dx\\&
\ \
+\displaystyle\int_{\mathbb{R}^{3}}\partial_{1}u_{2}\partial_{2}u_{2}\partial_{1}u_{2}dx
+\displaystyle\int_{\mathbb{R}^{3}}\partial_{2}u_{1}\partial_{1}u_{1}\partial_{2}u_{1}dx
+\displaystyle\int_{\mathbb{R}^{3}}\partial_{1}u_{1}\partial_{1}u_{1}\partial_{1}u_{1}dx\\&\
\
+\displaystyle\int_{\mathbb{R}^{3}}\partial_{2}u_{2}\partial_{2}u_{2}\partial_{2}u_{2}dx
+\displaystyle\int_{\mathbb{R}^{3}}\partial_{2}u_{2}\partial_{2}u_{1}\partial_{2}u_{1}dx\\
&
=\displaystyle-\int_{\mathbb{R}^{3}}(\partial_{2}u_{1}\partial_{1}u_{2}\partial_{3}u_{3}
+\partial_{3}u_{3}\partial_{1}u_{2}\partial_{1}u_{2}+\partial_{2}u_{1}\partial_{3}u_{3}\partial_{2}u_{1})
dx\vspace{1mm}\\
&\ \
-\displaystyle\int_{\mathbb{R}^{3}}(\partial_{1}u_{1}\partial_{1}u_{1}\partial_{3}u_{3}
+\partial_{3}u_{3}\partial_{2}u_{2}\partial_{2}u_{2}-\partial_{1}u_{1}\partial_{3}u_{3}\partial_{2}u_{2})
dx,
\end{array}$$
$$\begin{array}{ll}\displaystyle
I_{3}(t)&=\displaystyle \sum_{j=1}^{2}
\sum_{k=1}^{2}\displaystyle\int_{\mathbb{R}^{3}}u_{3}\partial_{3}u_{j}\partial_{kk}u_{j}dx\\&=\displaystyle-\sum_{j=1}^{2}
\sum_{k=1}^{2}\int_{\mathbb{R}^{3}}\partial_{k}u_{3}\partial_{3}u_{j}\partial_{k}u_{j}dx-\sum_{j=1}^{2}
\sum_{k=1}^{2}\int_{\mathbb{R}^{3}}u_{3}\partial_{3k}u_{j}\partial_{k}u_{j}dx\\
&=\displaystyle-\sum_{j=1}^{2}
\sum_{k=1}^{2}\int_{\mathbb{R}^{3}}\partial_{k}u_{3}\partial_{3}u_{j}\partial_{k}u_{j}dx+\frac{1}{2}\displaystyle\sum_{j=1}^{2}
\sum_{k=1}^{2}\int_{\mathbb{R}^{3}}\partial_{3}u_{3}\partial_{k}u_{j}\partial_{k}u_{j}dx,
\end{array}$$
$$\begin{array}{ll}\displaystyle
I_{4}(t)&=\displaystyle
\sum_{i,k=1}^{3}\int_{\mathbb{R}^{3}}u_{i}\partial_{i}u_{3}\partial_{kk}u_{3}dx\\&=\displaystyle-
\sum_{i,k=1}^{3}\int_{\mathbb{R}^{3}}\partial_{k}u_{i}\partial_{i}u_{3}\partial_{k}u_{3}dx-
 \sum_{i,k=1}^{3}\int_{\mathbb{R}^{3}}u_{i}\partial_{ik}u_{3}\partial_{k}u_{3}dx\\
& =\displaystyle-
\sum_{i,k=1}^{3}\int_{\mathbb{R}^{3}}\partial_{k}u_{i}\partial_{i}u_{3}\partial_{k}u_{3}dx.
\end{array}$$
The above four  equalities imply  that
 $$
|I_{i}|\leq C \int_{\mathbb{R}^{3}}|u_{3}||\nabla u||\Delta u|dx,\
i=1,3,4.
 $$
For $I_{2}$, we have
$$\begin{array}{ll}\displaystyle
|I_{2}|&=\displaystyle|\displaystyle\int_{\mathbb{R}^{3}}\partial_{3}u_{1}\partial_{1}u_{1}\partial_{3}u_{1}
+\partial_{3}u_{2}\partial_{2}u_{1}\partial_{3}u_{1}+\partial_{3}u_{3}\partial_{3}u_{1}\partial_{3}u_{1}\\
\displaystyle&\ \ \ \
+\partial_{3}u_{1}\partial_{1}u_{2}\partial_{3}u_{2}+\partial_{3}u_{2}\partial_{2}u_{2}\partial_{3}u_{2}
+\partial_{3}u_{3}\partial_{3}u_{2}\partial_{3}u_{2}dx|\\
\displaystyle&\leq\displaystyle C\int_{\mathbb{R}^{3}}|u_{3}||\nabla
u||\Delta u|dx+C\int_{\mathbb{R}^{3}}|\nabla_{h}u||\nabla u|^{2}dx.
\end{array}$$
Thus, above inequalities imply  that
\begin{equation} \label{12}\begin{array}{ll}\displaystyle \frac{1}{2}\frac{d}{dt}\|\nabla
u\|_{2}^{2}+\nu\|\Delta u\|_{2}^{2}&\leq \displaystyle
C\int_{\mathbb{R}^{3}}|u_{3}||\nabla u||\Delta
u|dx+C\int_{\mathbb{R}^{3}}|\nabla_{h}u||\nabla u|^{2}dx.\\
\displaystyle&=J_{1}(t)+J_{2}(t).
\end{array}\end{equation}
Similar to \eqref{8}, by Young's  and H$\ddot{\mbox{o}}$lder's
inequalities, we have
\begin{equation} \label{13}\begin{array}{ll}
J_{1}(t)&\leq C\displaystyle\| u_{3}\|_{q}^{\frac{r-1}{r}}\|
\partial_{1}u_{3}\|_{\alpha}^{\frac{1}{r}}\|\nabla u\|_{2}^{\frac{r-2}{r}}
\|\Delta u\|_{2}^{\frac{r+2}{r}},\vspace{2mm}\\
&\leq C\displaystyle\| u_{3}\|_{q}^{\frac{2(r-1)}{r-2}}\|
\partial_{1}u_{3}\|_{\alpha}^{\frac{2}{r-2}}\|\nabla
u\|_{2}^{2}+ \frac{\nu}{4}\|\Delta u\|_{2}^{2},
\end{array}\end{equation}
and \begin{equation} \label{14}J_{2}(t)\leq
C\|\nabla_{h}u\|_{2}\|\nabla u\|_{4}^{2}\leq
C\|\nabla_{h}u\|_{2}\|\nabla u\|_{2}^{\frac{1}{2}}\|\nabla_{h}\nabla
u\|_{2}\|\Delta u\|_{2}^{\frac{1}{2}}.\end{equation} Integrating
\eqref{12} and combing \eqref{1}, \eqref{7}, \eqref{13} and
\eqref{14}, we obtain
\begin{equation} \label{15}\begin{array}{ll}
&\|\nabla
u\|_{2}^{2}+\displaystyle\frac{7\nu}{4}\displaystyle\int_{0}^{t}\|\Delta
u\|_{2}^{2}d\tau\\
\displaystyle &\hspace{0.3cm} \leq\|\nabla
u(0)\|_{2}^{2}+\displaystyle C\int_{0}^{t}\|
u_{3}\|_{q}^{\frac{2(r-1)}{r-2}}\|
\partial_{1}u_{3}\|_{\alpha}^{\frac{2}{r-2}}\|\nabla
u\|_{2}^{2}d\tau\\
&\hspace{0.5cm} +(\sup_{0\leq s\leq t}\|\nabla_{h}
u\|_{2})(\displaystyle\int_{0}^{t}\|\nabla
u\|_{2}^{2}d\tau)^{\frac{1}{4}}\\&\ \ \ \ \
\times(\displaystyle\int_{0}^{t}\|\nabla_{h}\nabla
u\|_{2}^{2}d\tau)^{\frac{1}{2}}(\displaystyle\int_{0}^{t}\|\Delta
u\|_{2}^{2}d\tau)^{\frac{1}{4}}\\
\displaystyle &\hspace{0.3cm} \leq\|\nabla
u(0)\|_{2}^{2}+\displaystyle C\int_{0}^{t}\|
u_{3}\|_{q}^{\frac{2(r-1)}{r-2}}\|
\partial_{1}u_{3}\|_{\alpha}^{\frac{2}{r-2}}\|\nabla
u\|_{2}^{2}d\tau\\
&\hspace{0.5cm} +C\displaystyle[\int_{0}^{t}\|
u_{3}\|_{q}^{\frac{2(r-1)}{r-2}}\|
\partial_{1}u_{3}\|_{\alpha}^{\frac{2}{r-2}}\|\nabla
u\|_{2}^{2}d\tau]^{\frac{r-2}{r-1}}\times
 [\int_{0}^{t}\|\Delta
 u\|_{2}^{2}d\tau]^{\frac{1}{r-1}+{\frac{1}{4}}}\\
 &\hspace{0.5cm} + \displaystyle\|\nabla_{h}
u(0)\|_{2}^{2}(\displaystyle\int_{0}^{t}\|\Delta
u\|_{2}^{2}d\tau)^{\frac{1}{4}}.
\end{array}\end{equation}
 If $ 3r-7>0$,  we have $$\frac{1}{r-1}+{\frac{1}{4}}<1,
$$
then, by H$\ddot{\mbox{o}}$lder's and Young's inequalities, from
\eqref{15} we get
$$\begin{array}{ll}
&\|\nabla
u\|_{2}^{2}+\displaystyle\frac{7\nu}{4}\displaystyle\int_{0}^{t}\|\Delta
u\|_{2}^{2}d\tau\\
\displaystyle &\leq\|\nabla u(0)\|_{2}^{2}+C+\displaystyle
C\int_{0}^{t}\| u_{3}\|_{q}^{\frac{2(r-1)}{r-2}}\|
\partial_{1}u_{3}\|_{\alpha}^{\frac{2}{r-2}}\|\nabla
u\|_{2}^{2}d\tau\\
&\ \ +C\displaystyle[\int_{0}^{t}\|
u_{3}\|_{q}^{\frac{2(r-1)}{r-2}}\|
\partial_{1}u_{3}\|_{\alpha}^{\frac{2}{r-2}}\|\nabla
u\|_{2}^{2}d\tau]^{\frac{4(r-2)}{3r-7}}+\frac{3\nu}{4}\int_{0}^{t}\|\Delta
u\|_{2}^{2}d\tau \vspace{1mm}\\
\displaystyle &\leq\displaystyle C\int_{0}^{t}\|
u_{3}\|_{q}^{\frac{8(r-1)}{3r-7}}\|
\partial_{1}u_{3}\|_{\alpha}^{\frac{8}{3r-7}}\|\nabla
u\|_{2}^{2}d\tau\times[\int_{0}^{t}\|\nabla
u\|_{2}^{2}d\tau]^{\frac{r-1}{3r-7}}\\
 &\ \ + \|\nabla
u(0)\|_{2}^{2}+C+\displaystyle C\int_{0}^{t}\|
u_{3}\|_{q}^{\frac{2(r-1)}{r-2}}\|
\partial_{1}u_{3}\|_{\alpha}^{\frac{2}{r-2}}\|\nabla
u\|_{2}^{2}d\tau+\frac{3\nu}{4}\int_{0}^{t}\|\Delta
u\|_{2}^{2}d\tau.\\
\end{array}$$
 Absorbing the last term and  applying \eqref{1}, it follows that
$$\begin{array}{ll}
\|\nabla
u\|_{2}^{2}+\displaystyle{\nu}\displaystyle\int_{0}^{t}\|\Delta
u\|_{2}^{2}d\tau&\leq\displaystyle\displaystyle
C\int_{0}^{t}\|u_{3}\|_{q}^{\frac{8(r-1)}{3r-7}}\|
\partial_{1}u_{3}\|_{\alpha}^{\frac{8}{3r-7}}\|\nabla
u\|_{2}^{2}d\tau+\|\nabla u(0)\|_{2}^{2}\\
&\ \ \ \ \ +\displaystyle
C\int_{0}^{t}\|u_{3}\|_{q}^{\frac{2(r-1)}{r-2}}\|
\partial_{1}u_{3}\|_{\alpha}^{\frac{2}{r-2}}\|\nabla
u\|_{2}^{2}d\tau+C,
\end{array}
$$
therefore we prove \eqref{10}. If $3r-7=0$,  we have
$$\frac{1}{r-1}+{\frac{1}{4}}=1,$$
then by Young's inequality and \eqref{15} we have
$$\begin{array}{ll}
\|\nabla
u\|_{2}^{2}&+\displaystyle\frac{5\nu}{4}\displaystyle\int_{0}^{t}\|\Delta
u\|_{2}^{2}d\tau\vspace{2mm}\\
\displaystyle &\leq\|\nabla u(0)\|_{2}^{2}+\displaystyle
C\int_{0}^{t}\|u_{3}\|_{q}^{8}\|
\partial_{1}u_{3}\|_{\alpha}^{6}\|\nabla
u\|_{2}^{2}d\tau+ C\\
&\ \ +C\displaystyle[\int_{0}^{t}\|u_{3}\|_{q}^{8}\|
\partial_{1}u_{3}\|_{\alpha}^{6}\|\nabla
u\|_{2}^{2}d\tau]^{\frac{1}{4}}\times
 \int_{0}^{t}\|\Delta
 u\|_{2}^{2}d\tau,
\end{array}$$
which shows that the proof of (ii) is completed. \par Finally, we
prove \eqref{c95} in a similar way, integrating \eqref{12} and using \eqref{1},
\eqref{100}, \eqref{13} and \eqref{14},
we obtain
\begin{equation} \label{93}\begin{array}{ll}
&\|\nabla
u\|_{2}^{2}+\displaystyle\frac{\nu}{2}\displaystyle\int_{0}^{t}\|\Delta
u\|_{2}^{2}d\tau\\
\displaystyle &\hspace{0.3cm} \leq\|\nabla
u(0)\|_{2}^{2}+\displaystyle C\int_{0}^{t}\|
u_{3}\|_{q}^{\frac{2(r-1)}{r-2}}\|
\partial_{3}u_{3}\|_{\alpha}^{\frac{2}{r-2}}\|\nabla
u\|_{2}^{2}d\tau\\
&\hspace{0.5cm} +(\sup_{0\leq s\leq t}\|\nabla_{h}
u\|_{2})(\displaystyle\int_{0}^{t}\|\nabla
u\|_{2}^{2}d\tau)^{\frac{1}{4}}\\&\ \ \ \ \
\times(\displaystyle\int_{0}^{t}\|\nabla_{h}\nabla
u\|_{2}^{2}d\tau)^{\frac{1}{2}}(\displaystyle\int_{0}^{t}\|\Delta
u\|_{2}^{2}d\tau)^{\frac{1}{4}}\\
\displaystyle &\hspace{0.3cm} \leq\|\nabla
u(0)\|_{2}^{2}+\displaystyle C\int_{0}^{t}\|
u_{3}\|_{q}^{\frac{2(r-1)}{r-2}}\|
\partial_{3}u_{3}\|_{\alpha}^{\frac{2}{r-2}}\|\nabla
u\|_{2}^{2}d\tau\\
&\hspace{0.5cm} +C\displaystyle\int_{0}^{t}\|
u_{3}\|_{q}^{\frac{2(r-1)}{r-2}}\|
\partial_{3}u_{3}\|_{\alpha}^{\frac{2}{r-2}}\|\nabla
u\|_{2}^{2}d\tau\times
 [\int_{0}^{t}\|\Delta
 u\|_{2}^{2}d\tau]^{\frac{1}{4}}\\
 &\hspace{0.5cm} + \displaystyle\|\nabla_{h}
u(0)\|_{2}^{2}(\displaystyle\int_{0}^{t}\|\Delta
u\|_{2}^{2}d\tau)^{\frac{1}{4}}.
\end{array}\end{equation}
By using H$\ddot{\mbox{o}}$lder's and Young's inequalities, we immediately
have
\begin{equation*} \label{92}\begin{array}{ll}
\|\nabla
u\|_{2}^{2}+\displaystyle\frac{\nu}{2}\displaystyle\int_{0}^{t}\|\Delta
u\|_{2}^{2}d\tau &\leq\|\nabla u(0)\|_{2}^{2}+\displaystyle
C\int_{0}^{t}\| u_{3}\|_{q}^{\frac{2(r-1)}{r-2}}\|
\partial_{3}u_{3}\|_{\alpha}^{\frac{2}{r-2}}\|\nabla
u\|_{2}^{2}d\tau\\
&\hspace{0.5cm} +C\displaystyle\int_{0}^{t}\|
u_{3}\|_{q}^{\frac{8(r-1)}{3(r-2)}}\|
\partial_{3}u_{3}\|_{\alpha}^{\frac{8}{3(r-2)}}\|\nabla
u\|_{2}^{2}d\tau+C.
\end{array}\end{equation*}
Therefore,  we complete the proof of Lemma \ref{l2.3}.
\end{proof}
\section{{Proof of Main Results}}
\label{} In this section, we  prove our main results.
\par\textbf{Proof of Theorem 1.1}   The
framework of the proof is standard, we refer to \cite{[3]}. Without
loss of  generality, in the proof, we will assume that $j=1, k=3,$
the other cases can be discussed in the same way (for details see
Remark 3.1 below).
\par  It is  well known that there exists a unique
strong solution  $u$ for a short time interval if $u_{0}\in V$. In
addition, this strong solution  $u\in C([0,T^{*});V)\cap
L^{2}(0,T^{*}; H^{2}(\mathbb{R}^{3}))$ is the only weak solution
with the initial datum $u_{0}$, where $(0,T^{*})$ is the maximal
interval of existence of the unique strong solution. If $T^{*}\geq
T,$ then there is nothing to prove. If, on the other hand, $T^{*}<
T,$ then our strategy is to show that the $H^{1}$ norm of this
strong solution is bounded uniformly in time over the interval
$(0,T^{*})$, provided condition \eqref{e} is valid. As a result the
interval $(0,T^{*})$ can not be a maximal interval of existence, and
consequently $T^{*}\geq T,$ which concludes our proof. \par In order
to prove the $H^{1}$ norm of the strong solution $u$ is bounded on
interval $(0,T^{*})$, combing with the energy equality \eqref{1}, it
is sufficient to prove
\begin{equation} \label{a0}
\|\nabla
u\|_{2}^{2}+\displaystyle{\nu}\displaystyle\int_{0}^{t}\|\Delta
u\|_{2}^{2}d\tau\leq C,\ \forall \ t\in(0, T^{*}),
\end{equation}
where the constant $C$ depends on $T$, $K_{1} $ and $M$.

\par Firstly,  by energy inequality \eqref{1}, we have
\begin{equation} \label{a3}
\|u_{3}\|_{ L^{\infty}(0,T; L^{2}(\mathbb{R}^{3}))}\leq
C,\end{equation} where $C$ depends only on  $K_{1}$.  Then we show
\eqref{a0} is true on a small interval $(0,t_1)$ with some
$0<t_1<T^*$, because the constant $C$ in \eqref{a0} depends only on
$ K_{1}$ and $ M$, we give the same process to treat $t_1$ as the start
point. After finite steps, we get \eqref{a0} holds true on the whole
interval $(0,T^{*})$.
\par Now, by using  of Lemma \ref{l2.1}, Lemma \ref{l2.2} and Lemma
\ref{l2.3} with $ \alpha =3, \sigma=\frac{9}{8}$ and $q=2, $ then form
\eqref{4} and \eqref{6}, we have
\begin{equation}\label{ggg} r=\frac{7}{3}, \ s=\frac{2}{3}.\end{equation}
 Applying \eqref{a3} and
\eqref{ggg},  \eqref{11} becomes
\begin{equation} \label{a4}\begin{array}{ll}
\|\nabla
u\|_{2}^{2}&+\displaystyle\frac{5\nu}{4}\displaystyle\int_{0}^{t}\|\Delta
u\|_{2}^{2}d\tau\\ \displaystyle &\leq C
\displaystyle\int_{0}^{t}\|u_{3}\|_{2}^{8}\|
\partial_{1}u_{3}\|_{3}^{6}\|\nabla
u\|_{2}^{2}d\tau+\|\nabla u(0)\|_{2}^{2}+ C\\& \ \ \ \
+C\displaystyle[\int_{0}^{t}\|u_{3}\|_{2}^{8}\|
\partial_{1}u_{3}\|_{3}^{6}\|\nabla
u\|_{2}^{2}d\tau]^{\frac{1}{4}}\times
 \int_{0}^{t}\|\Delta
 u\|_{2}^{2}d\tau.\\
\end{array}\end{equation}
In view of \eqref{e}, we can choose $0<t_{1}<T^{*}$ small enough such
that
$$
C\|u_{3}\|_{ L^{\infty}(0,T;
L^{2}(\mathbb{R}^{3}))}^{2}\|\partial_{1}u_{3}\|_{L^{\infty}(0,T;
L^{3}(\mathbb{R}^{3}))}^{\frac{3}{2}}[\int_{0}^{t_{1}}\|\nabla
u\|_{2}^{2}d\tau]^{\frac{1}{4}}\leq\frac{\nu}{4}.
$$
 Then \eqref{a4}
becomes
\begin{equation} \label{a5}
\sup_{0\leq t\leq t_{1}}\|\nabla
u\|_{2}^{2}+\displaystyle{\nu}\displaystyle\int_{0}^{t_{1}}\|\Delta
u\|_{2}^{2}d\tau\leq\|\nabla u(0)\|_{2}^{2}+C.\end{equation} From
\eqref{a5}, we have
$$\|\nabla_{h}u(t_{1})\|_{2}<\|\nabla u(0)\|_{2}+C,\ \ \|\nabla u(t_{1})\|_{2}<\|\nabla u(0)\|_{2}+C,
$$ so we can repeat the above argument with initial value at
$t_{1}$ to obtain the similar estimate \eqref{7} in Lemma
\ref{l2.2}, and we have, for $t_{1}<t<T^{*}$,
$$
\begin{array}{ll}
 \displaystyle &\|\nabla_{h}u\|_{2}^{2}+\displaystyle\nu\int_{t_1}^{t}\|\nabla_{h}\nabla
u\|_{2}^{2}d\tau\vspace{2mm}\\
 \displaystyle &\ \ \ \ \ \leq
C\displaystyle[\int_{t_1}^{t}\|u_{3}\|_{2}^{8}\|
\partial_{1}u_{3}\|_{3}^{6}\|\nabla
u\|_{2}^{2}d\tau]^{\frac{1}{4}}\times
 [\int_{t_1}^{t}\|\Delta u\|_{2}^{2}d\tau]^{\frac{3}{4}}\vspace{2mm}\\
 &\ \ \ \ \ \ \ \ +\displaystyle\|\nabla_{h}
u(t_1)\|_{2}^{2}.\vspace{2mm}\\
 \end{array}
$$
From above inequality, we  obtain a similar estimate as \eqref{a4},
for $t_{1}<t<T^{*}$,
$$\begin{array}{ll}
\|\nabla
u\|_{2}^{2}&+\displaystyle\frac{5\nu}{4}\displaystyle\int_{t_1}^{t}\|\Delta
u\|_{2}^{2}d\tau\\ \displaystyle &\leq C
\displaystyle\int_{t_1}^{t}\|u_{3}\|_{2}^{8}\|
\partial_{1}u_{3}\|_{3}^{6}\|\nabla
u\|_{2}^{2}d\tau+\|\nabla u(t_1)\|_{2}^{2}+ C\\& \ \ \ \
+C\displaystyle[\int_{t_1}^{t}\|u_{3}\|_{2}^{8}\|
\partial_{1}u_{3}\|_{3}^{6}\|\nabla
u\|_{2}^{2}d\tau]^{\frac{1}{4}}\times
 \int_{t_1}^{t}\|\Delta
 u\|_{2}^{2}d\tau.\\
\end{array}$$
There exists a number  $t_{2}$  such that
$$
C\|u_{3}\|_{ L^{\infty}(0,T;
L^{2}(\mathbb{R}^{3}))}^{2}\|\partial_{1}u_{3}\|_{L^{\infty}(0,T;
L^{3}(\mathbb{R}^{3}))}^{\frac{3}{2}}[\int_{t_{1}}^{t_{2}}\|\nabla
u\|_{2}^{2}d\tau]^{\frac{1}{4}}\leq\frac{\nu}{4},
$$
and we have
$$\sup_{ t_{1}\leq t\leq t_{2}} \|\nabla
u\|_{2}^{2}+\displaystyle{\nu}\displaystyle\int_{t_{1}}^{t_{2}}\|\Delta
u\|_{2}^{2}d\tau\leq\|\nabla u(t_{1})\|_{2}^{2}+C\leq  \|\nabla
u(0)\|_{2}^{2}+C.
$$
Then we can repeat the above process from $t_{2}$, if $t_{2}<
T^{*}$. Actually, since $\partial_{1}u_{3}\in L^{\infty}(0,T;
L^{3}(\mathbb{R}^{3}))$, and the coefficients involving
$\|\partial_{1}u_{3}\|_{L^{\infty}(0,T; L^{3}(\mathbb{R}^{3}))}$,
depend only on $ K_{1}$ and $ M$, after finite steps of the process
of the bootstrap iteration, we can get an estimate on the whole time
interval
$$
\|\nabla
u\|_{2}^{2}+\displaystyle{\nu}\displaystyle\int_{0}^{t}\|\Delta
u\|_{2}^{2}d\tau\leq\|\nabla u(0)\|_{2}^{2}+C,
$$
for all $t\in(0,T^{*})$. Therefore, the $ H^1$ norm of the strong
solution $u$ is bounded on the maximal interval of existence $(0,
T^{*})$. This completes the proof of Theorem \ref{t1.1}.
\begin{rem}In above proof, we note that if we give the
additional assumptions  on $\partial_{3}u_{3}$, namely, we choose
$j=3, k=3, $ then the inequality \eqref{5} may be replaced by
$$
\begin{array}{ll}
 \displaystyle
 \frac{1}{q}\frac{d}{dt}\|u_{3}\|_{q}^{q}&+C(q)\nu\|\nabla|u_{3}|^{\frac{q}{2}}\|_{2}^{2}\\
 \displaystyle &=\displaystyle\| u\|_{2}^{s}\|\nabla
u\|_{2}^{\frac{8}{3}-s}\|u_{3}\|_{q}^{q-2}\|\partial_{3}u_{3}\|_{\alpha}^{{1}/{3}},
\end{array}
$$
and the inequality \eqref{8} may be replaced by
$$
\begin{array}{ll}
 \displaystyle \frac{1}{2}\frac{d}{dt}\|\nabla_{h}u\|_{2}^{2}&+\nu\|\nabla_{h}\nabla u\|_{2}^{2}\vspace{2mm}\\
\displaystyle &\leq C\displaystyle\| u_{3}\|_{q}^{\frac{r-1}{r}}\|
\partial_{3}u_{3}\|_{\frac{q}{q-r+1}}^{\frac{1}{r}}\|\nabla u\|_{2}^{\frac{r-2}{r}}
 \|\partial_{1}\nabla u\|_{2}^{\frac{1}{r}} \|\partial_{2}\nabla u\|_{2}^{\frac{1}{r}}
\|\nabla_{h}\nabla u\|_{2},\vspace{2mm}\\
\displaystyle &\leq C\displaystyle\| u_{3}\|_{q}^{2}\|
\partial_{3}u_{3}\|_{\alpha}^{\frac{2}{r-1}}\|\nabla u\|_{2}^{\frac{2(r-2)}{r-1}}
 \|\Delta u\|_{2}^{\frac{2}{r-1}} +\frac{\nu}{2}
\|\nabla_{h}\nabla u\|_{2}^{2},
\end{array}
$$
and \eqref{7} becomes
$$
\begin{array}{ll}
 \displaystyle \|\nabla_{h}u\|_{2}^{2}&+\displaystyle\nu\int_{0}^{t}\|\nabla_{h}\nabla
u\|_{2}^{2}d\tau\vspace{2mm}\\
\displaystyle &\leq
C\displaystyle[\int_{0}^{t}\|u_{3}\|_{q}^{\frac{2(r-1)}{r-2}}\|
\partial_{3}u_{3}\|_{\alpha}^{\frac{2}{r-2}}\|\nabla
u\|_{2}^{2}d\tau]^{\frac{r-2}{r-1}}\times
 [\int_{0}^{t}\|\Delta u\|_{2}^{2}d\tau]^{\frac{1}{r-1}}\vspace{2mm}\\
 &\ \ \ \ \ + C\displaystyle\|\nabla_{h}
u(0)\|_{2}^{2}.
 \end{array}
$$
The other cases can be done in a similar way. Since if we want to
provide the conditions on $\partial_{j}u_{2}, \ j=1,2,3$, in Lemma
\ref{l2.1}, we will use $|u_{2}|^{q-2}u_{2},$ as test function in
the equation for $u_{2},$  and we can get the similar results to
that in Lemma \ref{l2.2} and Lemma \ref{l2.3}. Therefore, this
method is suitable for every one of the nine  entries of the gradient
tensor.\end{rem}
\par \textbf{Proof of Theorem \ref{t1.3}} We take different strategy to prove Theorem \ref{t1.3}
(i) and (ii) in turn. The framework of  our proof  of  Theorem
\ref{t1.3}  is  also  standard. As to the second part of this
theorem, we give another inequality on $u_3$, and then we prove \eqref{a0}.\\
$\bullet j\ne k$
\par To prove Theorem \ref{t1.3} (i), we take the same strategy as
that in  Theorem \ref{t1.1}, and $(0,T^{*})$ is the maximal interval
of existence of the strong solution.  Next, we show \eqref{a0} is
true under the condition of \eqref{z}-\eqref{f}. Similar to the
proof of Theorem \ref{t1.1}, we take $j=1, k=3$, and prove the
boundedness of $u_3$ in $L^{\infty}(0,T; L^{q})$ with some $q$ at
first,  then apply Lemma \ref{l2.3}  and Gronwall's inequality to
get \eqref{a0}.
\par For
$$\alpha\in(3, \infty),
$$
we choose the parameters in the following form
\begin{equation}\label{b86}\left\{\begin{array}{l}
\displaystyle\frac{1}{\sigma_1}=\frac{(7-\frac{3}{\alpha})+\sqrt{\frac{9}{\alpha^2}-\frac{12}{\alpha}+103}}{18},\\
\displaystyle q_1=\frac{6\alpha\sigma_1}{3\alpha+\sigma_1},\\
\displaystyle s_1=\frac{3-2\sigma_1}{\sigma_1},\\
\displaystyle r_1=\frac{6(\alpha-1)\sigma_1}{3\alpha+\sigma_1}+1.
\end{array}
\right.
\end{equation}
Then, the above parameters satisfy \eqref{3} and \eqref{4}, namely,
$$
\frac{2}{q_{1}}=\frac{1}{3{\alpha}}+\frac{1}{\sigma_{1}}\ \
\mbox{and} \ \ \frac{2}{3s_{1}-2}=\frac{2\sigma_{1}}{9-8\sigma_{1}}.
$$
We choose
\begin{equation}\label{b85}
\beta=\frac{2\sigma_1}{9-8\sigma_1},
\end{equation}
then we have $$\frac{1}{\sigma_1}=\frac{2}{9}(\frac{1}{\beta}+4),$$
 and \eqref{b86}, we get  $\frac{3}{2\alpha}+\frac{2}{\beta}=f(\alpha)$. We denote
\begin{equation}\label{b89}
V_{1}(t)=\int_{0}^{t}\|
\partial_{1}u_{3}\|_{\alpha}^{\beta}\|\nabla
u\|_{2}^{2}d\tau=\int_{0}^{t}\|
\partial_{1}u_{3}\|_{\alpha}^{\frac{2\sigma_1}{9-8\sigma_1}}\|\nabla
u\|_{2}^{2}d\tau.
\end{equation}
By  the following fact
\begin{equation}\label{a30}\begin{array}{ll}
\displaystyle\frac{18\sigma-5\sigma^{2}}{3\sigma^{2}+14\sigma-18}&\displaystyle-\frac{\sigma}{\sigma-1}
\displaystyle
=\displaystyle\frac{(-8\sigma+9)\sigma^{2}}{(3\sigma^{2}+14\sigma-18)(\sigma-1)}>0\vspace{2mm}\\
\displaystyle &\ \ \ \ \ (\mbox{when} \displaystyle
\frac{\sqrt{103}-7}{3}=\sigma_{2}<\sigma<9/8),
\end{array}\end{equation}
where $\sigma_2=\frac{\sqrt{103}-7}{3}$ is the positive solution of
$3\sigma^{2}+14\sigma-18=0$. Therefore,  for every $\alpha$, we have
(in fact, from \eqref{b86} we get
$\alpha=\frac{18\sigma_1-5\sigma^{2}_{1}}{3(3\sigma_{1}^{2}+14\sigma_{1}-18)})$
\begin{equation}\label{a31}\alpha>\frac{\sigma_1}{3(\sigma_1-1)}.\end{equation}
Applying \eqref{b86} and \eqref{a31}, we get
\begin{equation}\label{d32}
\frac{2}{q_{1}}<\frac{\sigma_{1}-1}{\sigma_{1}}
+\frac{1}{\sigma_{1}}=1\Longrightarrow q_{1}>2,
\end{equation}
thus \eqref{2} is satisfied with $\alpha, \sigma_1$ and $ q_1$,  and
for $s_{1}$, we have $$\frac{2}{3}<s_1<1.$$ Integrating \eqref{4}
with $q=q_1, \sigma=\sigma_1$ and $s=s_1$, in view of H\"{o}lder's
inequality and \eqref{1}, we get
\begin{equation}\label{a10}
\begin{array}{ll}
 \displaystyle
\|u_{3}\|_{q_1}^{2}&\leq\|u_{3}(0)\|_{q_1}^{2}+\displaystyle
C\int_{0}^{t} \|\nabla
u\|_{2}^{\frac{8}{3}-s_1}\|\partial_{1}u_{3}\|_{\alpha}^{{1}/{3}}d\tau\\
&\leq\|u_{3}(0)\|_{q_1}^{2}+\displaystyle C[\int_{0}^{t} \|\nabla
u\|_{2}^{2}d\tau]^{\frac{8-3s_1}{6}}
[\int_{0}^{t}\|\partial_{1}u_{3}\|_{\alpha}^{\frac{2}{3s_1-2}}d\tau]^{\frac{3s_1-2}{6}}\\
&\leq\|u_{3}(0)\|_{q_1}^{2}+\displaystyle C
[\int_{0}^{t}\|\partial_{1}u_{3}\|_{\alpha}^{\frac{2}{3s_1-2}}d\tau]^{\frac{3s_1-2}{6}}\\
&=\|u_{3}(0)\|_{q_1}^{2}+\displaystyle C
[\int_{0}^{t}\|\partial_{1}u_{3}\|_{\alpha}^{\frac{2\sigma_1}{9-8\sigma_1}}d\tau]^{\frac{9-8\sigma_1}{6\sigma_1}}.
\end{array}\end{equation}
From \eqref{b86}, we have
$\alpha=\frac{18\sigma_1-5\sigma^{2}_{1}}{3(3\sigma_{1}^{2}+14\sigma_{1}-18)}$,
 it is not difficult to see that
$k(\sigma)=\frac{18\sigma-5\sigma^{2}}{3(3\sigma^{2}+14\sigma-18)}$
is a decreasing  function with respect to the variable $\sigma$, and
\begin{equation}\label{d24}
\lim_{\sigma\rightarrow{\frac{9}{8}^{-}}}\frac{18\sigma-5\sigma^{2}}{3(3\sigma^{2}+14\sigma-18)}=3,
\ \mbox{and}\
\lim_{\sigma\rightarrow{\sigma_{2}^{+}}}\frac{18\sigma-5\sigma^{2}}{3(3\sigma^{2}+
14\sigma-18)}=\infty.\end{equation}  By \eqref{d24}, we have that
$\sigma_{1}$ satisfies  $\sigma_{2}<\sigma_{1}<\frac{9}{8}$.
Together with $3<\alpha<\infty$, we obtain $2<q_1<\frac{9}{4}$. From
$u_{0}\in V$, by Sobolev embedding we have $\|u_{3}(0)\|_{q_1}<C$
for some $C>0$.  By virtue of
\begin{equation}\label{a13}
\frac{4(\sigma+3{\alpha})}{3(3\sigma-2){\alpha}-11\sigma}
-\frac{2\sigma}{9-8\sigma}
=\frac{2[-5\sigma^{2}+18\sigma+3(-3\sigma^{2}-14\sigma+18)
{\alpha}]}{(3(3\sigma-2){\alpha}-11\sigma)(9-8\sigma)},
\end{equation}
we have
\begin{equation}\label{b95}\frac{4(\sigma_{1}+3{\alpha})}{3(3\sigma_{1}-2){\alpha}-11\sigma_{1}}
=\frac{2\sigma_{1}}{9-8\sigma_{1}},
\end{equation}
where $3(3\sigma_{1}-2){\alpha}-11\sigma_{1}>0$ (see
\eqref{a15} and \eqref{b97} below).
 Therefore, applying
\eqref{a10} and \eqref{b95},  we obtain
\begin{equation}\label{a51}
\begin{array}{ll}\|u_{3}\|_{q_1}^{2}&\leq\|u_{3}(0)\|_{q_1}^{2}+\displaystyle C
[\int_{0}^{t}\|\partial_{1}u_{3}\|_{\alpha}^{\frac{2\sigma_1}{9-8\sigma_1}}d\tau]^{\frac{9-8\sigma_1}{6\sigma_1}}.\\
\end{array}\end{equation}
From the condition \eqref{f}, we have
\begin{equation}\label{a19}
\|u_{3}\|_{ L^{\infty}(0,T; L^{q_{1}}(\mathbb{R}^{3}))}\leq
C,\end{equation} where $C$ depend  on $T, K_{1}, M$.     The
selected $r_1$ in \eqref{b86} satisfies
\begin{equation}\label{b88}
r_{1}=\frac{(q_{1}+1){\alpha}-q_{1}}{{\alpha}},\end{equation} and
from \eqref{3}, \eqref{4} and \eqref{6}, we have
\begin{equation}\label{b98}
\frac{8}{3r_1-7}=\frac{4(\sigma_1+3{\alpha})}{3(3\sigma_1-2){\alpha}-11\sigma_1}.
\end{equation}  Because
\begin{equation}\label{a15}\begin{array}{ll}
\displaystyle\frac{\sigma}{\sigma-1}\displaystyle-\frac{11\sigma}{3\sigma-2}
\displaystyle
&=\displaystyle\frac{(-8\sigma+9)\sigma}{(\sigma-1)(3\sigma-2)}>0,
\displaystyle \ \forall  \ \displaystyle 1<\sigma<9/8,
\end{array}\end{equation}
 and take into account of \eqref{a31}, \eqref{a15} implies that
\begin{equation}\label{b97}\alpha> \frac{11\sigma_1}{3(3\sigma_1-2)},
\ \mbox{namely}\ 3(3\sigma_1-2){\alpha}-11\sigma_1>0,\end{equation}
and then $\frac{8}{3r_1-7}>0$ in \eqref{b98}. From \eqref{b88}, the
fact that  $3<\alpha<\infty$ and $q_1>2$, one has
\begin{equation}\label{a8}
r_1>\frac{7}{3}.
\end{equation}
By Lemma 2.3, we have
\begin{equation}\label{a11}\begin{array}{ll}
\|\nabla
u\|_{2}^{2}&+\displaystyle{\nu}\displaystyle\int_{0}^{t}\|\Delta
u\|_{2}^{2}d\tau\\ &\leq\displaystyle\displaystyle
C\int_{0}^{t}\|u_{3}\|_{q_1}^{\frac{8(r_1-1)}{3r_1-7}}\|
\partial_{1}u_{3}\|_{\alpha}^{\frac{8}{3r_1-7}}\|\nabla
u\|_{2}^{2}d\tau+\|\nabla u(0)\|_{2}^{2}\\
&\ \ \ \ \ +\displaystyle
C\int_{0}^{t}\|u_{3}\|_{q_1}^{\frac{2(r_1-1)}{r_1-2}}\|
\partial_{1}u_{3}\|_{\alpha}^{\frac{2}{r_1-2}}\|\nabla
u\|_{2}^{2}d\tau+C.
\end{array}\end{equation}
Applying \eqref{a11}   and \eqref{a19}, and the fact
$\frac{8}{3r-7}>\frac{2}{r-2}$  for all $r$ satisfying \eqref{9}, we
have
$$\begin{array}{ll}
\|\nabla
u\|_{2}^{2}+\displaystyle{\nu}\displaystyle\int_{0}^{t}\|\Delta
u\|_{2}^{2}d\tau\leq\displaystyle\displaystyle C V_{1}(t)+\|\nabla u(0)\|_{2}^{2}+C.\\
\end{array}
$$
The result for $\alpha > 3$ follows from Gronwall's inequality, and this end the proof of (i).\\
$\bullet j=k$
\par For  Theorem \ref{t1.3} (ii), without loss of generality,  we
 assume $j=3, k=3$. For every
\begin{equation}\label{d1} \alpha\in(\frac{9}{5},\infty),\end{equation}
we set
\begin{equation}\label{d2}\left\{\begin{array}{l}
\displaystyle \frac{1}{\mu}=\frac{1-\frac{12}{\alpha}+\sqrt{\frac{144}{\alpha^2}-\frac{264}{\alpha}+289}}{24},\\
\displaystyle q_2=\frac{2\alpha\mu}{\alpha+\mu},\vspace{2mm}\\
\displaystyle r=\frac{2\mu\alpha-\mu+\alpha}{\alpha+\mu}.
\end{array}
\right.
\end{equation}
 From \eqref{d2}, we have
\begin{equation}\label{d3}
\alpha=\frac{12\mu+5\mu^2}{6\mu^2+\mu-12},
\end{equation}
and note that $h(\mu):=\frac{12\mu+5\mu^2}{6\mu^2+\mu-12}$ is
a decreasing function of $\mu$, and
\begin{equation}\label{d4}
\lim_{\mu\rightarrow{3^{-}}}\frac{12\mu+5\mu^2}{6\mu^2+\mu-12}=\frac{9}{5}
\ \mbox{and}\
\lim_{\mu\rightarrow{{\frac{4}{3}}^{+}}}\frac{12\mu+5\mu^2}{6\mu^2+\mu-12}=\infty.\end{equation}
By \eqref{d1} and \eqref{d4}, we have
\begin{equation}\label{d5}\frac{4}{3}<\mu<3,\end{equation}
and \eqref{d2} follows
\begin{equation}\label{d6} \frac{1}{\mu}+\frac{1}{\alpha}+\frac{q_2-2}{q_2}=1.\end{equation}
On the other hand,
$$
\frac{12\mu+5\mu^2}{6\mu^2+\mu-12}-\frac{\mu}{\mu-1}=\frac{\mu^2(6-\mu)}{(6\mu^2+\mu-12)(\mu-1)}>0
\ \mbox{with} \ \frac{4}{3}<\mu<3.
$$
Combing \eqref{d3} and above inequality, we have
$\alpha>\frac{\mu}{\mu-1}$, and hance \eqref{d2} implies  $q_2>2.$
We choose
\begin{equation}\label{d7}
\beta=\frac{2\mu}{3-\mu},
\end{equation}
then we have
$$\frac{1}{\mu}=\frac{2}{3}(\frac{1}{\beta}+\frac{1}{2}),$$
by \eqref{d2} and \eqref{d7}, we have
$\frac{3}{2\alpha}+\frac{2}{\beta}=g(\alpha)$. We denote
\begin{equation}\label{d8}
V_{2}(t)=\int_{0}^{t}\|
\partial_{3}u_{3}\|_{\alpha}^{\beta}\|\nabla
u\|_{2}^{2}d\tau=\int_{0}^{t}\|
\partial_{3}u_{3}\|_{\alpha}^{\frac{2\mu}{3-\mu}}\|\nabla
u\|_{2}^{2}d\tau.
\end{equation} Next, we give another estimates on
$u_3$. We use $|u_{3}|^{q_2-2}u_{3}$  as test function  in the
equation \eqref{a} for $u_{3}.$ By using of Gagliardo-Nirenberg and
H$\ddot{\mbox{o}}$lder's inequalities, and applying  the inequality
\eqref{p}, we have
\begin{equation}\label{d9}
\begin{array}{ll}
 \displaystyle \frac{1}{q}\frac{d}{dt}\|u_{3}\|_{q_2}^{q_2}&+C(q_2)\nu\|\nabla|u_{3}|^{\frac{q_2}{2}}\|_{2}^{2}\\
 &=\displaystyle
 -\int_{\mathbb{R}^{3}}\partial_{3}p|u_{3}|^{q_2-2}u_{3}dx\displaystyle  \vspace{2mm}\\
 &\leq C\displaystyle\int_{\mathbb{R}^{3}}|p||u_{3}|^{q_2-2}|\partial_{3}u_{3}|dx\displaystyle  \vspace{2mm}\\
 \displaystyle &\leq\displaystyle
  C\|p\|_{\mu}\|u_{3}\|_{q_2}^{q_2-2}\|\partial_{3}u_{3}\|_{{\alpha}}\ \ \displaystyle(\mbox{by} \eqref{d6} ) \vspace{2mm}\\
 \displaystyle &\leq\displaystyle   C \|u
 \|_{2\mu}^2\|u_{3}\|_{q_2}^{q_2-2}\|\partial_3u_{3}\|_{\alpha}\displaystyle
 \ (\mbox{by }\eqref{p})\vspace{2mm}\\
 \displaystyle &\leq\displaystyle C\| u\|_{2}^{\frac{3-\mu}{\mu}}\|\nabla
 u\|_{2}^{\frac{3(\mu-1)}{\mu}}
 \|u_{3}\|_{q_2}^{q_2-2}\|\partial_{3}u_{3}\|_{\alpha}.
 \vspace{2mm}
\end{array}
\end{equation}
The above inequality immediately implies that
\begin{equation}\label{d10}
\begin{array}{ll}
\displaystyle
\frac{1}{2}\frac{d}{dt}\|u_{3}\|_{q_2}^{2}\displaystyle
\leq\displaystyle C\| u\|_{2}^{\frac{3-\mu}{\mu}}\|\nabla
 u\|_{2}^{\frac{3(\mu-1)}{\mu}}\|\partial_{3}u_{3}\|_{\alpha}.
\end{array}
\end{equation}
In view of  \eqref{d5}, we have $\frac{3(\mu-1)}{\mu}<2$, applying
Young's inequality, we have
\begin{equation}\label{d11}
\begin{array}{ll}
\displaystyle
\frac{1}{2}\frac{d}{dt}\|u_{3}\|_{q_2}^{2}\displaystyle
\leq\displaystyle C\|\nabla
 u\|_{2}^{2}+\| u\|_{2}^{2}\|\partial_{3}u_{3}\|_{\alpha}^{\frac{2\mu}{3-\mu}}.
\end{array}
\end{equation}
Integrating \eqref{d11} on time, and by energy inequality \eqref{1},
we obtain
\begin{equation}\label{d12}
\begin{array}{ll}
\displaystyle \|u_{3}\|_{q_2}^{2}\displaystyle \leq\displaystyle
\|u_{3}(0)\|_{q_2}^{2}+C\int_{0}^{t}\|\partial_{3}u_{3}\|_{\alpha}^{\frac{2\mu}{3-\mu}}d\tau.
\end{array}
\end{equation}
By the condition \eqref{zz2}, $\eqref{d2}$ and
\eqref{d5}, we have  $q_2<6$. Note that
$\|u_{3}(0)\|_{q_2}<C$ for some $C>0$, we get
\begin{equation}\label{d13}
u_3\in L^{\infty}(0, T; L^{q_2}(\mathbb{R}^{3})).
\end{equation}
Keeping in mind that  we have  another estimates in Lemma \ref{l2.3}
\begin{equation} \label{d14}\begin{array}{ll}
&\|\nabla
u\|_{2}^{2}+\displaystyle\frac{\nu}{2}\displaystyle\int_{0}^{t}\|\Delta
u\|_{2}^{2}d\tau\\
&\ \ \ \ \ \  \leq\|\nabla u(0)\|_{2}^{2}+\displaystyle
C\int_{0}^{t}\| u_{3}\|_{q_2}^{\frac{2(r-1)}{r-2}}\|
\partial_{3}u_{3}\|_{\alpha}^{\frac{2}{r-2}}\|\nabla
u\|_{2}^{2}d\tau\\
&\hspace{1.3cm} +C\displaystyle\int_{0}^{t}\|
u_{3}\|_{q_2}^{\frac{8(r-1)}{3(r-2)}}\|
\partial_{3}u_{3}\|_{\alpha}^{\frac{8}{3(r-2)}}\|\nabla
u\|_{2}^{2}d\tau+C.
\end{array}\end{equation}
By \eqref{d2}, we have $r=\frac{(q_{2}+1){\alpha}-q_{2}}{{\alpha}},$
and $r>2,$ therefore
$$
\frac{8}{3(r-2)}=\frac{8(\mu+\alpha)}{3(2\mu\alpha-3\mu-\alpha)}>0.
$$
Moreover, from \eqref{d3}, we have (note that
$2\mu\alpha-3\mu-\alpha>0$)
$$
\frac{8(\mu+\alpha)}{3(2\mu\alpha-3\mu-\alpha)}-\frac{2\mu}{3-\mu}=\frac{2[12\mu+5\mu^2+(-6\mu^2-\mu+12)\alpha]}{3(2\mu\alpha-3\mu-\alpha)(3-\mu)}=0.
$$
 Combing \eqref{d13} and
\eqref{d14}, and the fact $\frac{2}{r-2}<\frac{8}{3(r-2)},$ we have
$$\begin{array}{ll}
\|\nabla
u\|_{2}^{2}+\displaystyle{\nu}\displaystyle\int_{0}^{t}\|\Delta
u\|_{2}^{2}d\tau\leq\displaystyle\displaystyle C V_{2}(t)+\|\nabla u(0)\|_{2}^{2}+C,\\
\end{array}
$$
and end the proof for $
\alpha\in\left(\frac{9}{5},\infty\right)$ by  using of  Gronwall's inequality.

\section*{Acknowledgement}
The second author would like to thank Dr. Ting Zhang   for his
helpful suggestions. This work is supported partially by NSFC
10931007, and Zhejiang NSF of China Z6100217.

\bibliographystyle{elsarticle-num}

\newpage
\begin{figure} \centering
\includegraphics[width=0.8\textwidth,height=0.3\textheight]{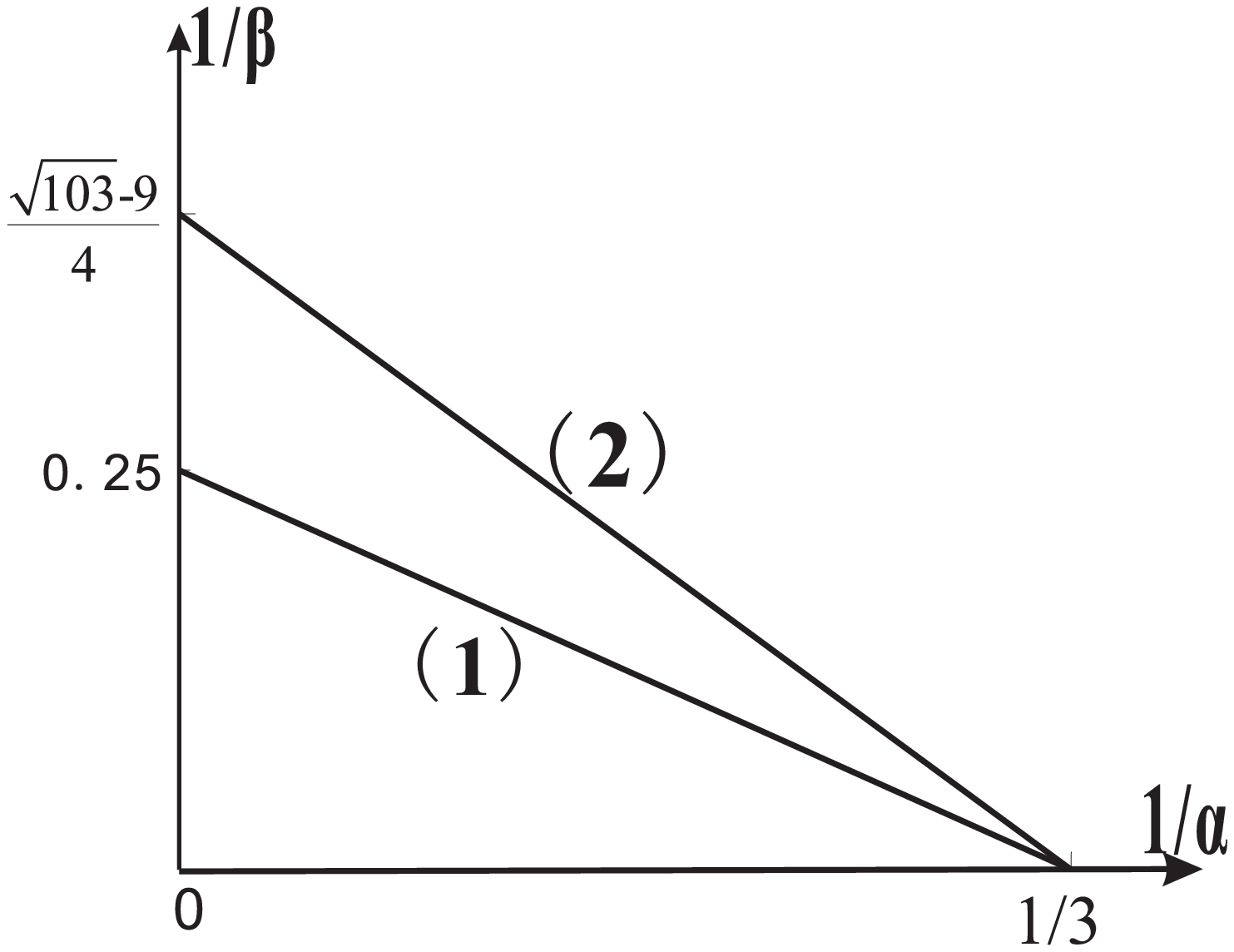}
\caption{ Case of $j\neq k$}\label{fig:128} The line "\textbf{(1)}"
is the result of C.S. Cao, E.S. Titi in \cite{[2]}, which signifies
\eqref{c}. The line "\textbf{(2)}" is our result, which means
\eqref{f}.
\end{figure}

\begin{figure}
\centering
\includegraphics[width=0.8\textwidth,height=0.35\textheight]{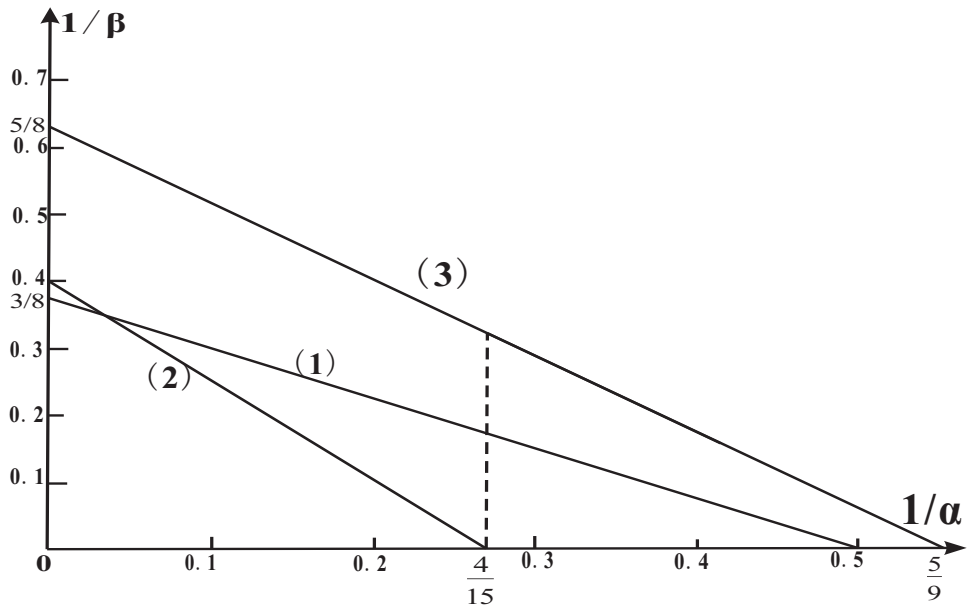}
\caption{Case of $j= k$} \label{fig:131}The line "\textbf{(1)}"
signifies \eqref{d}, which is also considered by  C.S. Cao, E.S.
Titi in \cite{[2]}. The line "\textbf{(3)}" is our result, which
mean  \eqref{zz2}.  The result of Y. Zhou, M. Pokorn$\acute{\mbox{y}
}$  in \cite{[26]} is showed by line "\textbf{(2)}".
\end{figure}

\end{document}